\documentclass[a4paper,fleqn]{cas-sc}

\usepackage[numbers]{natbib}

\def\tsc#1{\csdef{#1}{\textsc{\lowercase{#1}}\xspace}}
\tsc{WGM}
\tsc{QE}
\newcommand{\qednew}{\hfill {\small $\blacksquare$} \medskip}

\usepackage{algorithm}
\usepackage{algpseudocode}
\setcitestyle{square}

\usepackage{graphicx}
\usepackage{subcaption}

\usepackage{hyperref}

\hypersetup{colorlinks=true}

\hypersetup{colorlinks=true, linkcolor=blue, citecolor=blue,urlcolor=blue}

\usepackage{amsthm}

\newtheorem{theorem}{Theorem}[section]

\newdefinition{rmk}{Remark}
\newtheorem{corollary}[theorem]{Corollary}

\newtheorem{claim-proof}[thm]{Claim}
\newproof{pot}{Proof of Theorem \ref{thm}}

\begin{document}
\let\WriteBookmarks\relax
\def\floatpagepagefraction{1}
\def\textpagefraction{.001}

\shorttitle{Quadruple Roman Domination}

\shortauthors{V.S.R. Palagiri, G.P. Sharma, I.G. Yero}

\title [mode = title]{Complexity Issues Concerning the Quadruple Roman
Domination Problem in Graphs}



%

\author[1]{Venkata Subba Reddy Palagiri}

\cormark[1]

\fnmark[1]

\ead{pvsr@nitw.ac.in}




\author[1]{Guru Pratap Sharma}

\fnmark[2]

\ead{gpcs20111@student.nitw.ac.in}



\author[2]{Ismael G. Yero}

\fnmark[3]

\ead{ismael.gonzalez@uca.es}

\affiliation[1]{organization={Department of Computer Science and Engineering},
            addressline={},
            city={Hanamkonda},
            postcode={506004},
            state={Telangana},
            country={India}}

\affiliation[label2]{organization={Department of Mathematics, Universidad de Cádiz},
            addressline={Algeciras Campus},
            country={Spain}}

\cortext[1]{Corresponding author}

\fntext[1]{ORCID: \href{https://orcid.org/0000-0002-0972-1141}{0000-0002-0972-1141}}

\fntext[3]{ORCID: \href{https://orcid.org/0000-0002-1619-1572}{0000-0002-1619-1572}}


\begin{abstract}
Given a  graph $G$ with vertex set $V(G)$, a mapping $h : V(G) \rightarrow \lbrace 0, 1, 2, 3, 4, 5 \rbrace$ is called a quadruple Roman dominating function (4RDF) for $G$ if it holds the following. Every vertex $x$ such that $h(x)\in \{0,1,2, 3\}$ satisfies that $h(N[x]) = \sum_{v\in N[x]} h(v) \geq |\{y:y \in N(x) \; \text{and} \; h(y) \neq 0\}|+4$, where $N(x)$  and $N[x]$ stands for the open and closed neighborhood of $x$, respectively. The smallest possible weight $\sum_{x \in V(G)} h(x)$ among all possible 4RDFs $h$ for $G$ is the quadruple Roman domination number of $G$, denoted by  $\gamma_{[4R]}(G)$.

This work is focused on complexity aspects for the problem of computing the value of this parameter for several graph classes. Specifically, it is shown that the decision problem concerning $\gamma_{[4R]}(G)$ is NP-complete when restricted to star convex bipartite, comb convex bipartite, split and planar graphs. In contrast, it is also proved that such problem can be efficiently solved for threshold graphs where an exact solution is demonstrated, while for graphs having an efficient dominating set, tight upper and lower bounds in terms of the classical domination number are given. In addition, some  approximation results to the problem are given. That is, we show that the problem cannot be approximated within $(1 - \epsilon) \ln |V|$ for any $\epsilon > 0$ unless $P=NP$.  An approximation algorithm for it is proposed, and its APX-completeness proved, whether graphs of maximum degree four are considered. Finally, an integer linear programming formulation for our problem is presented.
\end{abstract}



\begin{keywords}
 Quadruple Roman domination\sep Roman domination\sep Graph classes\sep NP-complete\sep APX-complete
\end{keywords}

\ExplSyntaxOn
\keys_set:nn { stm / mktitle } { nologo }
\ExplSyntaxOff

\maketitle

\section{Introduction}

The topic of domination in graphs is a classical one in the area of graph theory. The amount of results, information, open questions, etc on this research is huge. One can easily see this fact by simply referring to the books on domination \cite{Bresar,Chartrand,hhh-1,hhh-2,hhh,Haynes2,Haynes,Hen-Yeo}, some of which have very recently appeared. Moreover, thousands of scientific articles can be found in the literature concerning the domination theory. Despite these facts, there are still numerous open questions that are nowadays not yet settled, as well as, many different ongoing research lines that are seeing their birth every day. Some classical problems on domination are the ones concerning defense strategies for the vertices of a graph. One of the most popular of such strategies is that known as Roman domination, which is commonly understood to have some roots in the ancient Roman Empire, and thus its historical name (see \cite{Revelle2000,Stewart1999}). We are aimed in this work to considering some computational aspects of a problem concerning one of these defense strategies. Our work combines combinatorial, computational and defense strategies that are nowadays usually of interest for the graph theory community.

Throughout our whole exposition, the graphs $G$ considered are undirected, connected and simple, and have vertex set $V(G)$ and edge set $E(G)$. Edges of $G$ are denoted by $uv$ where $u,v$ are vertices in $V(G)$. The \textit{open neighbourhood} of a vertex $v\in V(G)$ is \begin{math} N(v)=\{u\,|\,uv \in E(G)\} \end{math}, and the \textit{closed neighbourhood} is $N[v]=N(v) \cup \{v\}$. The maximum degree of $G$ is $\Delta(G)=\max\{|N(v)\,|\,v\in V(G)\}$. An \textit{induced subgraph} associated with a subset $S$ of vertices of $G$ is denoted by $G[S]$. A complete bipartite graph $K_{1,n}$ is called a \textit{star graph}, and the vertex of the star with degree $n$ is called the \textit{center}. A \textit{comb graph} is constructed from the path $P_n$ ($n\ge 2$) by adding a pendant vertex to each vertex of such path. If $G$ is a bipartite graph with bipartition sets $V_1$ and $V_2$, and there exists a tree $T$ with vertex set $V_2$ such that for each $v\in V_1$, the open neighborhood $N(v)$ of $v$ in $G$ induces a connected subgraph of $T$, then $G$ is called a \textit{tree convex bipartite graph}. Particular situations of this definition are as follows. The tree convex bipartite graph $G$ is called a \textit{comb} (resp. \textit{star}) \textit{convex bipartite graph} if $T$ is a comb (resp. star). For other graph theoretic terminology we refer to \cite{west}.

The classical concepts in domination theory are as follows. A set $D\subset V(G)$ of a graph $G$ is known to be a \textit{dominating set} of such graph if $ \cup_{u \in D} N[u]=V(G)$, i.e., every vertex which is not in the set $D$ has at least one neighbor in $D$. This classical concept seems to have been first presented in \cite{Ore}, and right now several different domination related structures have been described in such a way that they are all giving more insight into the classical one, as well as, between themselves, see the already mentioned books \cite{Bresar,Chartrand,hhh-1,hhh-2,hhh,Haynes2,Haynes,Hen-Yeo}.

As mentioned, the Roman domination notion was formally introduced in graph theory based on protection strategies used to defend the Roman Empire, see \cite{Cockayne}. That is, a Roman dominating function (RDF) of a graph $G$ is a function $g:V(G) \rightarrow \{0,1,2\}$ such that for every vertex $x\in V(G)$ with $g(x)=0$ there exists a neighbor $y\in N(x)$ with $g(y)=2$. After this seminal work, various Roman domination strategies have been introduced to protect a given set of vertices of a graph under different constraints, see for instances \cite{RD1,RD2}. One of the ideas for this variations was focused on increasing the maximum label used in the considered functions. For examples, in 2016 Chellali et al. \cite{Chellali} introduced the Roman $\{2\}$-domination; the triple Roman domination was introduced by Ahangar et al. in 2021 \cite{trd}; and quadruple Roman domination was introduced by Amjadi et al. in 2021 \cite{qrdp}. Several other similar variations are known, but we prefer to refer to just the ones closer to our work, which focuses on the latter one of these mentioned variants.

A function \begin{math} f  : V\rightarrow \{0,1,2,3,4,5\} \end{math} on the vertex set of a graph $G=(V(G),E(G))$ is called a \textit{quadruple Roman dominating function} (shortly 4RDF) for $G$, if every vertex $x$ such that $f(x)\in \{0,1,2, 3\}$ satisfies that $f(N[x]) = \sum_{v\in N[x]} f(v) \geq |\{y:y \in N(x) \; \text{and} \; f(y) \neq 0\}|+4$. The weight of a 4RDF is $f(V(G))=\sum_{v\in V(G)}f(v)$ and the \textit{quadruple Roman domination number} of $G$, denoted by \begin{math} \gamma_{4R}(G) \end{math}, is the minimum weight among all possible 4RDFs for $G$. We consider in this work the minimum quadruple Roman domination problem (shortly M4RDP) on graphs, which is that of determining $\gamma_{4R}(G)$ for a graph $G$, as well as, the decision version regarding it, i.e., the one next formally stated.

\begin{center}
\fbox{
	\parbox{0.9\textwidth}{
		\textsc{\bf Decision Version of Quadruple Roman Domination Problem (4RDP)} \\
		\textit{Instance}: A simple, undirected graph $G$ and an integer $l \geq 0$.\\
		\textit{Question}: Is $\gamma_{4R}(G) \leq l$?}}
\end{center}

We must recall that, in \cite{qrdp}, it has been proved that the 4RDP is NP-complete for chordal graphs and  bipartite graphs and a few open questions were left for further investigation. Also, upper and lower bounds on the quadruple Roman domination number of trees were obtained in \cite{qrdp1}. In connection with complexity and computational aspects of the 4RDP, we consider here the 4RDP in some subclasses of bipartite graphs and planar graphs, which shows that even introducing more restrictions to this problem, it still remains very challenging. In consequence, we also investigate the M4RDP from combinatorial and approximation points of view. In addition, in order to possibly develop in future some heuristics on the M4RDP, we describe at the end of our exposition, an integer linear programming formulation of the M4RDP.

\section{NP-completeness results}\label{complexity}

In order to present our results, we shall use the following well known NP-complete problem, as shown in \cite{ci}.

\begin{center}
\fbox{
	\parbox{0.97\textwidth}{
		\textsc{\bf EXACT-3-COVER (ETC)} \\
		\textit{Instance}: A finite set $X$ of cardinality $3r$ and a collection $C$ of subsets of $3$ elements of $X$.\\
		\textit{Question}: Does $C$ have a subset $C^ \prime$ of cardinality $r$ which contains all elements of $X$?}}
\end{center}

\noindent
First notice that given an integer $s$ and a function $f$ of any graph $G$, it can be verified in polynomial time whether $f$ is a 4RDF of weight at most $s$ in $G$. Hence, the inclusion of the 4RDP in the class NP is warranted for any of the graph classes we consider in our work.

As pointed out by one of the reviewers, we remark that in parallel (and independently) with our work, the NP-complete results proved for star and comb convex bipartite graphs from this section have been recently published in \cite{Valen}, as a part of the more general setting of $[k]$-Roman domination in graphs. However, the reductions made there are using the 3-SAT problem, while we show here that, the EXACT-3-COVER (ETC) can be also used in this direction. We are even able to find simpler arguments than the ones used in \cite{Valen}.

\begin{theorem} \label{theorem1}
The 4RDP is NP-complete	even when restricted to star convex bipartite graphs.
\end{theorem}

\begin{proof}
We consider an instance $X=\lbrace x_1, x_2, \ldots, x_{3r} \rbrace$ and $C=\lbrace C_1, C_2,\ldots, C_s \rbrace$ of ETC, and develop a reduction to an instance of the 4RDP, as given below. To this end, we construct a star convex bipartite graph $G_{r,s}$ in polynomial time to be used in the reduction.

For each $x_i\in X$, we include vertices $x_i$ and $y_i$, and for each $C_i \in C$ we create a vertex $C_i$. Then we add vertices $b$, $b_1, b_2, \ldots,b_{14r+6}$. We now form a star graph with vertex set $\{b,b_1,b_2,\ldots,b_{14r+6}\}$ with $b$ as the center vertex. Also, we create another star graph  with vertex set $\{b,C_1,C_2,\ldots,C_{s}\}$ and again $b$ as the center vertex. Next, we add edges between $x_i$ and $C_j$, if $x_i \in C_j$. Finally, to obtain $G_{r,s}$, we add an edge between $x_i$ and $y_i$ for every $x_i\in X$. Figure \ref{fig:sg-1} shows a sketch of the graph constructed in this way. Notice that such graph is bipartite.

\begin{figure}[htbp]
    \centering
    \begin{subfigure}[b]{0.35\textwidth}
        \centering
        \includegraphics[width=5cm,height=5cm]{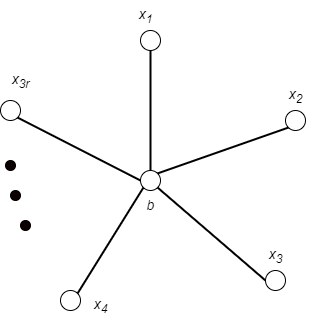}
        \caption{Star graph}
        \label{fig:scbp}
    \end{subfigure}%
    \hfill
    \begin{subfigure}[b]{0.6\textwidth}
        \centering
        \includegraphics[width=6cm,height=7cm]{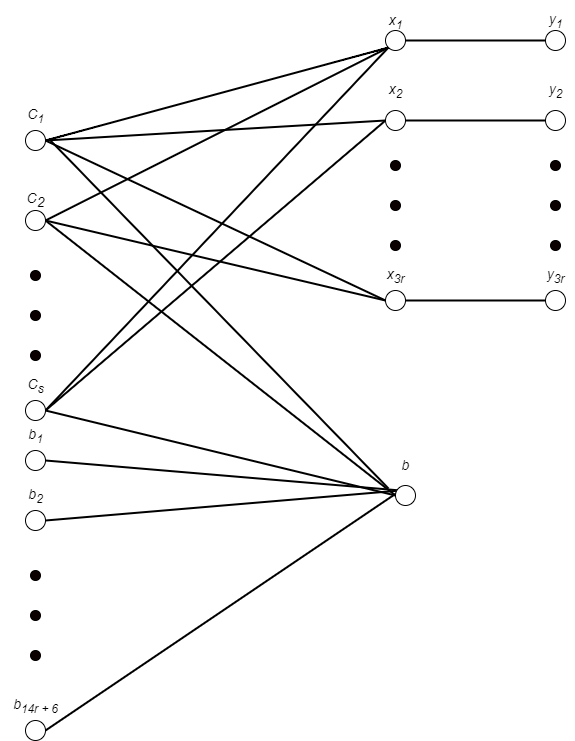}
        \caption{A sketch of the star convex bipartite graph $G_{r,s}$.}
        \label{fig:sg-1}
    \end{subfigure}
    \caption{Comparison of the Star graph and its corresponding star convex bipartite graph.}
    \label{fig:comparison}
\end{figure}

Let $P = \lbrace b, x_1, x_{2},\ldots,x_{3r}\rbrace$ and $Q = V(G_{r,s}) \setminus P$ (notice that $P$ and $Q$ are the bipartition sets of $G_{r,s}$). Consider now a star $S_{1,3r}$ on $3r+1$ vertices with vertex set $P$ as shown in Figure \ref{fig:scbp}. Observe that for any vertex $w\in Q$, the vertices of its open neighborhood $N(w)$ in $G_{r,s}$ (which is a subset of $P$) induce a star that is a subgraph of the defined star $S_{1,3r}$. Hence, the constructed graph $G_{r,s}$ is star convex bipartite. We next show that the corresponding instance of ETC has a positive solution if and only if $G_{r,s}$ has a 4RDF with weight at most $14r+5$.

Assume $C^\prime\subset C$ with cardinality $r$ represents a positive solution for ETC. We define a function $f : V(G_{r,s}) \rightarrow \{0, 1, 2,3,4,5\}$ with weight $14r+5$ as follows. 		
		\begin{equation}\label{stareq}
		f(u)=
		\begin{cases}
		5, & \text{if}\ u=b, \\
		4, & \text{if}\ u \in \{y_1,y_2,\ldots,y_{3r}\}, \\
		2, & \text{if}\ u \in C^\prime, \\
		0, & \text{otherwise}.
		\end{cases}
		\end{equation}	
We claim that $f$ is a quadruple Roman dominating function. That is as follows if we consider those vertices with labels at most three.
\begin{itemize}
    \item For any $i\in \{1,\dots,14r+6\}$, it holds $f(N[b_i]) = \sum_{v\in N[b_i]} f(v) = 5 = |\{y:y \in N(b_i) \; \text{and} \; f(y) \neq 0\}|+4$.
    \item For any $i\in \{1,\dots,3r\}$, it holds $f(N[x_i]) = \sum_{v\in N[x_i]} f(v) = 6 = |\{y:y \in N(x_i) \; \text{and} \; f(y) \neq 0\}|+4$.
    \item If $C_j\in C'$, then $f(N[C_j]) = \sum_{v\in N[C_j]} f(v) = 7 > 5 = |\{y:y \in N(C_j) \; \text{and} \; f(y) \neq 0\}|+4$.
    \item If $C_j\notin C'$, then $f(N[C_j]) = \sum_{v\in N[C_j]} f(v) = 5 = |\{y:y \in N(C_j) \; \text{and} \; f(y) \neq 0\}|+4$.
\end{itemize}
Thus, our claim follows, and we obtain that $\gamma_{4R}(G_{r,s})\le f(V(G_{r,s}))= 14r+5$ as desired.

On the other hand, let $t = 14r+5$, and consider a 4RDF $g$ for $G_{r,s}$ with the smallest possible weight at most $t$. Since $\gamma_{4R}(G_{r,s}) \leq t=14r+5$ and there are $14r+6$ vertices of type $b_i$, there must be at least one $b_j$ ($1\le j\le 14r+6$) such that $g(b_j)=0$. This implies that $g(b)=5$. Hence, based on this fact, if there is $b_k$ ($k\ne j$) such that $g(b_k)\ne 0$, then we can define a new 4RDF $g'$ such that $g'(x)=g(x)$ for any $x\in V(G_{r,s})\setminus\{v_k\}$ and $g'(b_k)=0$, which has smaller weight than $g$. This contradicts our assumption on the minimality of $g$. Thus,  $g(b) = 5$ and $g(b_i)= 0$ for every $1 \leq i \leq 14r+6$. We now require the following claim. From now on, whenever we write $x_i \in C_j$ we refer to the vertices $x_i$ and $C_j$ corresponding to the variable $x_i\in X$ that belong to the set $C_j\in C$.

\begin{claim-proof}
\label{claim-1}
For each $x_i \in X$ it holds $g(x_i) = 0$.
\end{claim-proof}

\noindent
\textbf{Proof of Claim \ref{claim-1}:}
Suppose that the number of vertices of type $x_i$ having positive label is $m\geq 1$. Hence, the number of $x_i$'s with label equal to zero is $3r-m$. Thus, one of the following situations occurs.
\begin{itemize}
    \item[(i)] $g(x_i) \neq 0$, which means $g(x_i)+g(y_i) \geq 5$.
    \item[(ii)] $g(x_i) = 0$, which leads to $g(y_i) \geq 4$.
\end{itemize}
Now, let $a = |\{y_i\,:\,g(y_i) = 4 \; \text{and} \; 1 \leq i \leq 3r\}|$. Clearly, for each $x_i$ with $g(x_i) = 0$ and $g(y_i) = 4$ there must be a vertex $C_j\in N(x_i)$ with $g(C_j) \geq 2$. Thus, the required number of vertices of type $C_j$ with $g(C_j) \geq 2$ is $\left\lceil \frac{a}{3} \right\rceil$ (notice that each $C_j$ has only three neighbors of type $x_i$). Consequently,
$$g(V(G_{r,s})) \geq 5+5m+4a+5(3r-m-a)+2\left\lceil \frac{a}{3} \right\rceil.$$
\noindent 
 Now, we consider three cases. \\
\textbf{Case 1:} $a \mod 3 =0$. It follows that $2\left\lceil \frac{a}{3} \right\rceil=2\frac{a}{3}$. Thus,
$$g(V(G_{r,s}))  \geq 5+5m+4a+5(3r-m-a)+2\left\lceil \frac{a}{3} \right\rceil
=14r+5+r-\frac{a}{3}.$$
\textbf{Case 2:} $a \mod 3 =1$. It follows that $2\left\lceil \frac{a}{3} \right\rceil=2\left(\frac{a+2}{3}\right)$. Hence
$$g(V(G_{r,s})) \geq 5+5m+4a+5(3r-m-a)+2\left\lceil \frac{a}{3} \right\rceil=14r+5+r-a+2\left(\frac{a+2}{3}\right)=14r+5+r-\frac{a}{3}+\frac{4}{3}.$$
\textbf{Case 3:} $a \mod 3 =2$. It follows that $2\left\lceil \frac{a}{3} \right\rceil=2\left(\frac{a+1}{3}\right)$. Hence
$$g(V(G_{r,s})) \geq 5+5m+4a+5(3r-m-a)+2\left\lceil \frac{a}{3} \right\rceil=14r+5+r-a+2\left(\frac{a+1}{3}\right)=14r+5+r-\frac{a}{3}+\frac{2}{3}.$$
\noindent Since $m \geq 1$ and $a \leq 3r-m$, it follows that $r-\frac{a}{3} >0$, $r-\frac{a}{3}+ \frac{4}{3}>0$, $r-\frac{a}{3}+ \frac{2}{3}>0$, and so, from the inequality above, we deduce $g(V) \geq 14r+5+r-\frac{a}{3} > 14r+5=t$, which is a contradiction. Thus, the proof of Claim \ref{claim-1} is completed.\qednew

\smallskip
From the claim above, it follows that  $g(y_i) \geq 4$ for every $1 \leq i \leq 3r$. Suppose there exist $y_k$ such that $g(y_k) = 5$, and let us count such vertices. Consider a vertex $x_i$ with $1 \leq i \leq 3r$. Since $g(x_i)=0$ (by the previous claim) and $g(y_i) \geq 4$, we deduce that
$$g(y_i) + \sum_{x_i\in C_j} g(C_j) = \sum_{v\in N[x_i]} g(v) = g(N[x_i])\ge |\{y:y \in N(x_i) \; \text{and} \; g(y) \neq 0\}|+4.$$

Note that any $x_i$ such that $g(C_j)=0$ for every $C_j$ with $x_i\in C_j$, satisfies that $g(y_i)= 5$. Moreover, for any vertex $x_i$ such that $g(C_j)\ne 0$ for some $C_j$ with $x_i\in C_j$ it holds that either $g(C_j)\ge 2$, or ($g(C_j)=1$ and $g(y_i)=5$). For otherwise (that is $g(C_j)=1$ and $g(y_i)=4$), we have $g(N[x_i]) = \sum_{v\in N[x_i]} g(v)=g(y_i)+g(C_j)=5<6 = |\{y:y \in N(x_i) \; \text{and} \; g(y) \neq 0\}|+4$, which is a contradiction. In this sense, since $g$ has weight at most $14r+5$, $g(b)=5$, and $g(y_i) \geq 4$ for every $1 \leq i \leq 3r$, the number of $y_i$ such that $g(y_i)=5$ is at most $K\le 2r$. This means there must be at least $R\ge r$ vertices $y_j$ such that $g(y_j)=4$. But then, the corresponding vertices $x_j$ adjacent to such $y_j$ (with $g(y_j)=4$) must have neighbors $C_k$ with $g(C_k)\ge 2$. Altogether, this leads to say that $g$ has weight $g(V(G_{r,s}))= 5+5K+4R+2\left\lceil R/3\right\rceil=5+5(3r-R)+4R+2\left\lceil R/3\right\rceil=5+14r+r-\left\lceil R/3\right\rceil>5+14r$, since $R\ge r$, and this is a contradiction. Hence, we have that for every $y_i$, with $1 \leq i \leq 3r$, it follows $g(y_i) = 4$.

As a consequence of the previous arguments, each $x_i$ must have a neighbor $C_j$ such that $x_i\in C_j$ with $g(C_j)=2$. Since  $g$ has weight at most $14r+5$, $g(b)=5$, and $g(y_i)=4$ for every $1 \leq i \leq 3r$, the number of vertices of type $C_j$ with $g(C_j)=2$ must be at least $r$. However, it must be indeed equal to $r$, since otherwise there will be a vertex $x_i$ that will have no neighbor $C_j$ such that $x_i\in C_j$ with $g(C_j)=2$. This means that each $x_i$ has exactly one such neighbor with label $2$ under $g$. Therefore, the set of such vertices $C_j$, such that $g(C_j)=2$, represents a positive solution for the ETC instance, and the proof is completed.
\end{proof}

We next consider another family of bipartite graphs where the 4RDP remains NP-complete.

\begin{theorem}
\label{theorem2}
The 4RDP is NP-complete	even when restricted to comb convex bipartite graphs.
\end{theorem}

\begin{proof}
We proceed similarly to Theorem \ref{theorem1} by reducing an instance $X=\lbrace x_1, x_2, \ldots, x_{3r} \rbrace$ and $C=\lbrace C_1, C_2,\ldots, C_s \rbrace$ of ETC to an instance of 4RDP. However, in this case, we need to consider $r\ge 4$ so that the reduction will work. To this end, we need to construct a  comb convex bipartite graph $H_{r,s}$ in polynomial time as follows.

\smallskip
For each $x_i \in X$, we include vertices $x_i,y_i,x_i^\prime,y_i^\prime$, and create $C_i$ for each $C_i \in C$. Also, we form a star graph with vertex set $\{b,b_1,b_2,\ldots,b_{26r+6}\}$ with $b$ as the center vertex; a second star graph with vertex set $\{b,C_1,C_2,\ldots,C_{s}\}$ with $b$ as the center vertex; and a third star graph  with vertex set $\{b^\prime,C_1,C_2,\ldots,C_{s}\}$ with $b^\prime$ as the center vertex. In addition, we add edges between $x_i$ and $C_j$ if $x_i \in C_j$; and add edges of the form $x_iy_i$ and $x_i^\prime y_i^\prime$. Finally, to obtain the graph $H_{r,s}$, we add all possible edges between $C_i$ and $x_j^\prime$. Figure \ref{fig:sg} shows a sketch of a graph constructed as detailed above.

\begin{figure}[ht]
    \centering
    \begin{subfigure}[b]{0.3\textwidth}
        \centering
        \includegraphics[width=4cm,height=6cm]{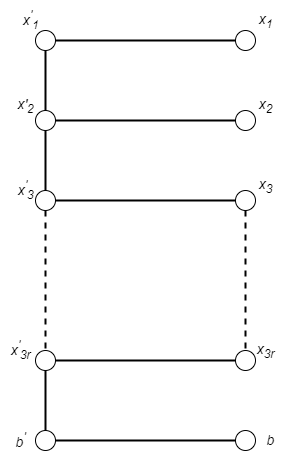}
        \caption{Comb graph}
        \label{fig:scobp}
    \end{subfigure}%
    \hfill
    \begin{subfigure}[b]{0.7\textwidth}
        \centering
        \includegraphics[width=7cm,height=8cm]{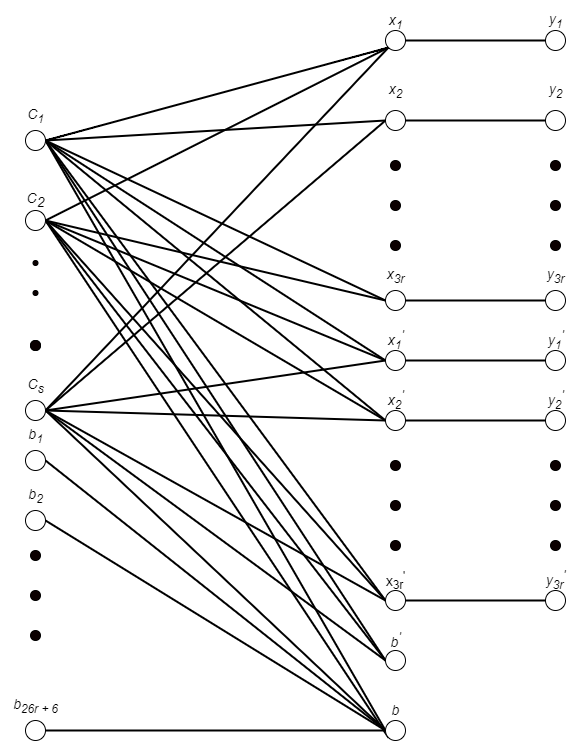}
        \caption{Comb convex bipartite graph construction from ETC.}
        \label{fig:sg}
    \end{subfigure}
    \caption{Comparison of Comb graph and convex bipartite construction.}
\end{figure}

Notice that $H_{r,s}$ is bipartite and let $P = \lbrace b, b^\prime x_1, x_{2},\ldots, x_{3r}, x_1^\prime, x_{2}^\prime, \ldots, x_{3r}^\prime \rbrace$ and $Q$ = $V  \setminus P$ be the bipartition sets of $H_{r,s}$. Consider the comb having vertex set $P$ as shown in Figure \ref{fig:scobp}. Hence, notice that the neighbors of each element in $Q$ induce a connected subgraph of such comb. Thus, the constructed graph $H_{r,s}$ is comb convex bipartite. We next show that an instance of the ETC problem has positive solution if and only if $H_{r,s}$ has a 4RDF with weight at most $26r+5$.

\smallskip
Assume $C^\prime$ with cardinality $r$ is a positive solution for our instance of the ETC problem. We define a function $f : V(H_{r,s}) \rightarrow \{0, 1, 2,3,4,5\}$ with weight $26r+5$ as follows.
\begin{equation} \label{combeq}
f(u)=
\begin{cases}
5, & \text{if}\ u=b, \\
4, & \text{if}\ u\in \{y_1,y_2,\ldots,y_{3r}\} \cup \{y_1^\prime,y_2^\prime,\ldots,y_{3r}^\prime\},\\
2, & \text{if}\ u \in C^\prime,\\
0, & \text{otherwise}.
\end{cases}
\end{equation}	
We claim that such $f$ is a 4RDF. To see this, we consider the following cases.
\begin{itemize}
    \item $f(N[b'])=\sum_{v\in N[b']} f(v) = 2r\ge r+4 = |\{y:y \in N(b') \; \text{and} \; f(y) \neq 0\}|+4$, since $r\ge 4$ as assumed.
    \item For any $i\in \{1,\dots,26r+6\}$, it holds $f(N[b_i]) = \sum_{v\in N[b_i]} f(v) = 5 = |\{y:y \in N(b_i) \; \text{and} \; f(y) \neq 0\}|+4$.
    \item For any $i\in \{1,\dots,3r\}$, it holds that $f(N[x_i]) = \sum_{v\in N[x_i]} f(v) = 6 = |\{y:y \in N(x_i) \; \text{and} \; f(y) \neq 0\}|+4$ as well as $f(N[x'_i]) = \sum_{v\in N[x'_i]} f(v) = 2r+4\ge r+5 = |\{y:y \in N(x'_i) \; \text{and} \; f(y) \neq 0\}|+4$, since $r\ge 4$.
    \item If $C_j\in C'$, then $f(N[C_j]) = \sum_{v\in N[C_j]} f(v) = 7 > 5 = |\{y:y \in N(C_j) \; \text{and} \; f(y) \neq 0\}|+4$.
    \item If $C_j\notin C'$, then $f(N[C_j]) = \sum_{v\in N[C_j]} f(v) = 5 = |\{y:y \in N(C_j) \; \text{and} \; f(y) \neq 0\}|+4$.
\end{itemize}
Therefore, $f$ is a 4RDF as claimed, and so,  $\gamma_{4R}(H_{r,s}) \leq f(V(H_{r,s}))=26r+5$.

\medskip
Conversely, if we assume that $H_{r,s}$ has a 4RDF $g$ with weight at most $26r+5$, then we can proceed with a proof that somehow runs along the lines of the second part of the proof of Theorem \ref{theorem1} to obtain the desired conclusion. That is, by using similar arguments to the mentioned ones, we first show that all the vertices $b_i$ with $1\le i\le 26r+6$ satisfy that $g(b_i)=0$ and that $g(b)=5$. We next prove that all the vertices $x_i,x'_i$ with $1\le i\le 3r$ satisfy that $g(x_i)=g(x'_i)=0$. That is stated in the next claim whose proof appears in the Appendix section, since it is similar to the one of Claim \ref{claim-1}.

\begin{claim-proof}
\label{claim-2}
For each $x_i \in X$ it holds $g(x_i) = 0$ and $g(x'_i)=0$.
\end{claim-proof}

We next proceed with proving that all the vertices $y_i$ with $1\le i\le 3r$ hold that $g(y_i)=4$, which is the next claim. The proof of this is relatively similar as the related one from Theorem \ref{theorem1}. The formal proof appears also in the Appendix section.

\begin{claim-proof}
\label{claim-3}
For each $x_i \in X$ it holds $g(y_i) = 4$.
\end{claim-proof}

Finally, and again similarly to the conclusion of the proof of Theorem \ref{theorem1}, by the previous arguments, each $x_i$ must have a neighbor $C_j$ such that $x_i\in C_j$ with $g(C_j)=2$. Since  $g$ has weight at most $26r+5$, $g(b)=5$, and $g(y_i)=4$ for every $1 \leq i \leq 3r$, the number of vertices of type $C_j$ with $g(C_j)=2$ is equal to $r$, since otherwise there will be a vertex $x_i$ that will have no neighbor $C_j$ such that $x_i\in C_j$ with $g(C_j)=2$. This again means that each vertex $x_i$ has exactly one such neighbor with label $2$ under $g$. Therefore, the set of these vertices $C_j$, such that $g(C_j)=2$, represents a positive solution for the ETC instance, and the proof is completed.
\end{proof}

\noindent We may remark that from Theorems \ref{theorem1} and \ref{theorem2} the below result follows.

\begin{corollary}
The 4RDP is NP-complete even when restricted to tree convex bipartite graphs.	
\end{corollary}

We next consider other classic graph class, in which we also prove that the 4RDP is NP-complete.

\begin{theorem}
\label{theoremsplit}
The 4RDP is NP-complete	even when restricted to split graphs.
\end{theorem}

\begin{proof}
In a similar manner as in the Theorems \ref{theorem1} and \ref{theorem2}, we next transform in polynomial time an instance $X=\lbrace x_1, x_2, \ldots, x_{3r} \rbrace$ and $C=\lbrace C_1, C_2,\ldots, C_s \rbrace$ of the ETC problem to an instance $G_{r,s}$ of 4RDP in a split graph.

\smallskip
For each $x_i \in X$, we include a vertex $x_i$ and  for each $C_i \in C$ we include vertex $C_i$. Then we add edges between every pair of $C_i,C_j\in C$. Finally, to obtain the graph $G_{r,s}$, we add edges of the form $x_iC_j$, if $x_i \in C_j$. Notice that $G_{r,s}$ is a split graph with a clique $\{C_1,C_2,\ldots,C_{s}\}$ and an independent set $V(G_{r,s}) \setminus \{C_1,C_2,\ldots,C_{s}\}$. We next show that a positive solution for the instance of ETC exists if and only if $G_{r,s}$ has a 4RDF with weight at most $5r$.

\smallskip
Assume $C^\prime$ with cardinality $r$ is a  solution for ETC. We define a mapping $f : V(G_{r,s}) \rightarrow \{0, 1, 2,3,4,5\}$ with weight $5r$ as follows.
\begin{equation} \label{eqsplit}
	f(u)=
	\begin{cases}
	5, & \text{if}\ u \in  C^\prime,\\
	0, & \text{otherwise}.
	\end{cases}
\end{equation}
We claim that such $f$ is a 4RDF. To see this, notice that $f(N[x_i])=\sum_{v\in N[x_i]} f(v) = 5 = |\{y:y \in N(x_i) \; \text{and} \; f(y) \neq 0\}|+4$ for every $x_i\in X$. Also, if $C_j\notin C'$, then $f(N[C_j])=\sum_{v\in N[C_j]} f(v) = 5r \ge r+4 = |\{y:y \in N(C_j) \; \text{and} \; f(y) \neq 0\}|+4$.
Thus, $f$ our claim follows and we have that $f(V(G_{r,s}))\le 5r$.

\medskip
Assume next that $g$ is a 4RDF of $G$ with weight at most $5r$.  For $i=0,1,2,3,4,5$, let $V_i =\{ v : g(v)=i \}$. We claim that $X \subset V_0$. To see this, suppose $l>0$ is the number of vertices in $X$ with a positive label under $g$. Notice that, if there is a vertex $x_i\in X$ with $g(x_i)=0$, then $x_i$ must have a neighbor $C_j\in C$ such that $g(C_j)=5$ since $x_i$ has no neighbors in $X$. Hence, it follows $g(V(G_{r,s})) \geq 5 \left\lceil \frac {3r-l}{3} \right\rceil +l+3 \left\lceil \frac{l}{3} \right\rceil > 5r$, which is a contradiction. Thus, $X \subset V_0$. Consequently, we deduce that $|\{C_i : g(C_i) =5\}| \geq r$ and  $|\{C_1,C_2,\ldots,C_{s}\} \cap V_5|=r$. Therefore $\{C_1,C_2,\ldots,C_{s}\} \cap V_5$ represents a positive solution to our instance of ETC.
\end{proof}	

We end this section with the following remark. If a graph $G$ is constructed from an instance $\langle X, C \rangle$ of ETC with $X=\lbrace x_1, x_2, \ldots, x_{3r} \rbrace$ and $C=\lbrace C_1, C_2,\ldots, C_s \rbrace$  such that $V(G)=\{x_1, x_2, \ldots, x_{3r}, C_1, C_2,\ldots, C_s\}$, $E(G)=\{x_iC_j : x_i \in C_j, 1\le i\le 3r, 1\le j\le s\}$ and the graph $G$ is planar, then $\langle X, C \rangle$ is called a \textit{Planar Exact Three Cover} (PETC) instance of $G$ and deciding if PETC has a solution is NP-complete  \cite{planarx3c}. Next we state the following theorem whose proof similar to the proof given in Theorem \ref{theoremsplit}.

\begin{theorem} \label{theorem3}
    The 4RDP is NP-complete	even when restricted to planar graphs.
\end{theorem}

\section{Bounds and exact values}

\noindent Let $G$ be a split graph whose vertex set has a split partition $(C, I)$, where $I$ is an independent set and $C$ is a clique. The graph $G$ is a \textit{threshold graph} if and only if (i) there is an ordering $(y_1, y_2, \ldots, y_q)$ of the vertices of $I$ such that $N_G(y_1) \supseteq N_G(y_2) \supseteq \cdots \supseteq N_G(y_q)$, and (ii) there is an ordering $(x_1, x_2, \ldots, x_p)$ of vertices of $C$ such that $N_G[x_1] \subseteq N_G[x_2] \subseteq \cdots \subseteq N_G[x_p]$ \cite{threshold}. 	

\begin{theorem}\label{th:threshold}
For any connected threshold graph $G$ with $|V(G)| \geq 2$, $\gamma_{[4R]}(G)=5$.
\end{theorem}

\begin{proof}
Let $(C, I)$ be a split partition of $G$ and let $(x_1, x_2, \ldots, x_p)$ be an ordering of the vertices of the clique $C$ such that $N_G[x_1] \subseteq N_G[x_2] \subseteq \cdots \subseteq N_G[x_p]$. From the definition of 4RDF, clearly $\gamma_{[4R]}(G) \geq 5$. Next, we define a function $f : V(G) \rightarrow \{0, 1, 2,3,4,5\}$ with weight $5$ as follows.
	\begin{equation}
	f(u)=
	\begin{cases}
	5, & \text{if}\ u = x_p,\\
	0, & \text{otherwise}.
	\end{cases}
	\end{equation}
It can be readily seen that such $f$ is a 4RDF since $x_p$ is adjacent to every vertex of $G$. Thus, $\gamma_{[4R]}(G) \leq 5$ and the result follows.
\end{proof}

In order to present some other exact values for the quadruple Roman domination number of some graphs, we need to introduce the following terminology. A dominating set $D=\{v_1,\dots,v_r\}\subset V(G)$ is an \textit{efficient dominating set} of the graph $G$, if the collection of sets $\{N[v_1],\dots,N[v_r]\}$ forms a partition of $V(G)$. Note that not every graph $G$ has an efficient dominating set. For instance, it can be readily seen that a cycle $C_n$ has an efficient dominating set if and only if $n$ is a multiple of $3$. In this sense, it is said that a graph $G$ is an \textit{efficient domination graph} whenever $G$ has an efficient dominating set. Efficient dominating sets in graphs were first introduced in \cite{Biggs} under the name of perfect codes, and have been further on very well studied. We may recall that, if $G$ is an efficient domination graph with efficient dominating set $D$, then $\gamma(G)=|D|$.

\begin{theorem}
\label{th:efficient}
For any efficient domination graph $G$, $$4\gamma(G)+1\le \gamma_{[4R]}(G) \leq 5\gamma(G).$$
\end{theorem}

\begin{proof}
Let $D=\{v_1,\dots,v_r\}\subset V(G)$ be an efficient dominating set of $G$. Let $f : V(G) \rightarrow \{0, 1, 2,3,4,5\}$ be a function on $V(G)$ such that  $f(u)=5$ if $u\in D$, and $f(u)=0$ otherwise. Since any vertex not in $D$ has exactly one neighbor in $D$, we observe that $f(N[u])=\sum_{v\in N[u]} f(v) = 5 = |\{y:y \in N(u) \; \text{and} \; f(y) \neq 0\}|+4$ for every $u\notin D$. Thus, $f$ is a 4RDF whose weight is $5\gamma(G)$, and the upper bound follows.

On the other hand, given the efficient dominating set $D=\{v_1,\dots,v_r\}$, we consider the partition $\{N[v_1],\dots,N[v_r]\}$ of $V(G)$. Let $g$ be a 4RDF for $G$ with weight $\gamma_{[4R]}(G)$. If there is a vertex $v_i\in D$ such that $g(v_i)\le 3$, then there must be a vertex $v'_i\in N[v_i]$ such that $g(v'_i)\ge 1$. This in addition leads to conclude that $g(N[v_i])\ge 5$ in such situation. Notice that in any other case, it must happen $g(N[v_j])\ge 4$.

Now, if $g(N[v_j])=4$ for every $1\le j\le 4$, then it must occur that $g(v_j)=4$ for every $1\le j\le 4$, and that $g(v)=0$ for each $v\in \overline{D}=V(G)\setminus D$. However, this means that no vertex $v\in \overline{D}$ satisfies the condition for $g$ to be a 4RDF, which is not possible. Thus, there must be at least one vertex $v_j\in D$ such that $g(N[v_j])\ge 5$. Therefore,
$$\gamma_{[4R]}(G)=g(V(G))=\sum_{v_i\in D}g(N[v_i])\ge 4|D|+1=4\gamma(G)+1,$$
which is the desired lower bound.
\end{proof}

The equality in the two bounds of Theorem \ref{th:efficient} happens in several well known non trivial graph classes. Some examples are next given.

Let $\mathcal{F}$ be a family of graphs $G_t$ defined as follows. We begin with $t\ge 2$ copies $H_1,\dots,H_t$ of any graph $H$, one graph $G'$ having a universal vertex $y$, $t$ vertices $x_1,\dots,x_{t}$, and an extra vertex $w$. Then we add edges between $x_i$ and all the vertices of the copy $H_i$ of $H$, where $1\le i\le t$. Next we add all edges between $w$ and all the vertices from all the copies of $H$, as well as all, from all the vertices of $G'$.

We readily observe that the set $D=\{y,x_1,\dots,x_{t}\}$ is an efficient dominating set of the graph $G_t\in \mathcal{F}$ of cardinality $t+1$. Consider now the function $f: V(G_t) \rightarrow \{0, 1, 2,3,4,5\}$ defined as follows.
$$f(u)=\begin{cases}
	5, & \text{if}\ u=w,\\
        4, & \text{if}\ u\in D\setminus\{y\}, \\
	0, & \text{otherwise}.
	\end{cases}$$
Note that every vertex with label $0$ under $f$ is adjacent either to a vertex with label $5$ or to two vertices, one labelled $4$ and other labelled $5$. Thus, $f$ is a 4RDF, and we have that $\gamma_{[4R]}(G_t)\le 4t+5=4(t+1)+1=4\gamma(G_t)+1$. Hence, the equality $\gamma_{[4R]}(G_t)=4\gamma(G_t)+1$ follows from Theorem \ref{th:efficient}.

\smallskip
Consider next a family $\mathcal{F}'$ of graphs $G'_r$ defined as next described. We begin with a graph $G$. Then we add $r\ge 5$ pendant vertices to each vertex of $G$, and subdivide each edge of $G$ with two vertices. Now, the set of vertices $V(G)$ (in the original graph $G$) forms an efficient dominating set of $G'_r$ of cardinality $|V(G)|$, which is indeed the unique dominating set of $G'_r$.

We also observe that $\gamma_{[4R]}(G'_t)\ge 5|V(G)|$ since each vertex $v\in V(G)$ is adjacent to $r\ge 5$ leaves. That is, if at least one of such leaves has label $0$ under any 4RDF, then the vertex of $V(G)$ adjacent to such leaf needs to have label $5$. Thus, the equality $\gamma_{[4R]}(G'_t)=5\gamma(G'_t)$ is obtained from Theorem \ref{th:efficient}.

\section{Approximation results}

Based on the NP-completeness of the 4RDP, it is desirable to obtain some approximation results on such problem. That is, we obtain a lower and an upper bounds on the approximation ratio of the M4RDP. We also show that the M4RDP is in APX-complete for graphs with maximum degree 4. To this end, we use the following known result from \cite{inapproxalg} on determining the domination number $\gamma(G)$ of a graph $G$, which is known as the \texttt{DOMINATION} problem. Recall that a set of vertices $D$ is a \textit{dominating set} of a graph $G$ if every vertex not in $D$ has a neighbor in $D$. Also, the \textit{domination number} $\gamma(G)$ of $G$ is the cardinality of a smallest dominating set of $G$.

\begin{theorem}{\em\cite{inapproxalg}}\label{inapprox}
For a graph $G$ and any $\delta > 0$,  the \texttt{DOMINATION} problem cannot have a solution with approximation ratio $(1 - \delta) \ln |V|$  unless $P = NP$.
\end{theorem}

We next prove a theorem which gives a lower bound on approximation ratio of any algorithm for M4RDP.

\begin{theorem} \label{inapproxmr3dp}
Given a graph $G$, M4RDP cannot have algorithm  with $(1 - \epsilon) \ln |V|$ approximation ratio for any $\epsilon > 0$, unless $P = NP$.
\end{theorem}

\begin{proof}
We prove the theorem by providing a reduction from the \texttt{DOMINATION} problem to M4RDP which preserves the approximation ratio. That is, an instance $G$ with vertex set $\{v_1,v_2,\ldots,v_{|V(G)|}\}$ of the \texttt{DOMINATION} problem is reduced to a corresponding instance $H$ with set of vertices $V(H)=\{v_1,v_2,\ldots,v_{|V(G)|},b_1,b_2,\ldots,b_{|V(G)|}\}$ and set of edges $E(H)=E(G) \cup \{v_1b_1, v_2b_2,\ldots,v_{|V(G)|}b_{|V(G)|}\}$ of M4RDP, see Figure \ref{gm} for an example of such construction. We first show the following necessary claim.
\begin{figure}[ht]
		\centering
		\includegraphics[scale=0.55]{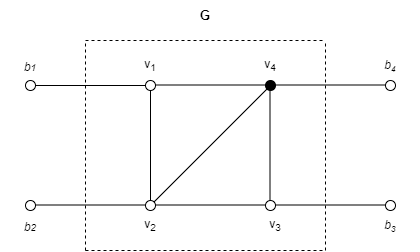}
		\caption{An example of the construction of $H$ from $G$}
		\label{gm}
\end{figure}		
	
\begin{claim-proof}\label{minimumcrdf}
If $H$ is the graph obtained from $G$ as described above, then $\gamma_{[4R]}(H) = 4|V(G)| + \gamma(G)$.
\end{claim-proof}

\noindent
\textbf{Proof of Claim \ref{minimumcrdf}:}
Let $S$	be a dominating set of $G$ with $|S|=\gamma(G)$. We define a function $f : V(H)\rightarrow \{0, 1, 2,3,4,5\}$ with weight $4|V(G)| + \gamma(G)$ as follows.
\begin{equation} 
	f(u)=
	\begin{cases}
	5, & \text{if}\ u=v_i \in  S,\\
    4, & \text{if }u=b_i \text{ for some $i$ such that } v_i \notin S,\\
	0, & \text{otherwise}.
	\end{cases}
\end{equation}
Notice that, if $v_i\notin S$, then there is at least one $v_j\in S$ such that $v_i\in N(v_j)$. Thus, $f(N[v_i])=\sum_{v\in N[v_i]} f(v) = 5|N(v_i)\cap S|+f(v_i) = 5|N(v_i)\cap S|+ 4 > |N(v_i)\cap S|+5 = |\{y:y \in N(v_i) \; \text{and} \; f(y) \neq 0\}|+4$. Also, for any $b_i$ such that $v_i\in S$, we have $f(N[b_i])=\sum_{v\in N[b_i]} f(v) = 5 = |\{y:y \in N(b_i) \; \text{and} \; f(y) \neq 0\}|+4$. Thus, it follows that $f$ is a 4RDF, and so $\gamma_{[4R]}(H) \leq 4|V(G)| + \gamma(G)$.

\smallskip
We now prove that $\gamma_{[4R]}(H) \geq 4|V(G)| + \gamma(G)$. Let $g$ be a 4RDF for the graph $H$. Consider the following function $h$ on $V(H)$ obtained from $g$ and satisfying that $h(V(H))=g(V(H))$.
\begin{equation} 
			h(u)=
			\begin{cases}
				0, & \text{ if }u = b_i, \; g(b_i)=5 \; \text{and} \; g(v_i)=0, \\
				5, & \text{ if }u = v_i, \; g(b_i)=5 \; \text{and} \; g(v_i)=0, \\
				g(u), & \text{otherwise}.
			\end{cases}
\end{equation}	
From the definition of a 4RDF the following constraints hold for each $1\le i\le n$.
\begin{itemize}
    \item[(i)] $h(b_i)+h(v_i) \geq 4$.
    \item[(ii)] If $h(b_i) \neq 4$, then $h(b_i)+h(v_i) \geq 5$.
    \item[(iii)] If $h(v_i) \neq 0$, then $h(b_i)+h(v_i) \geq 5$.
\end{itemize}
Let $V^0=\{v_i : h(b_i)=4  \; \text{and} \; h(v_i)=0\}$. Since $h(v_i)=0$ for all $v_i \in V^0$, from the definition of a 4RDF and the fact that $h(b_i)=4$, it follows that every $v_i \in V^0$ must be adjacent to a vertex with a positive label from the set $R=\{v_1,v_2,\ldots,v_{|V(G)|}\} \setminus V^0$.  It is readily verified that for all $v_i \in R$ it holds, $h(v_i) \neq 0$, and so, $h(b_i)+h(v_i) \geq 5$, if $v_i \in R$. Since each vertex in $V^0$ has a neighbor in $R$, it follows that $R$ is a dominating set of $G$. Moreover, from the constraints (i) and (iii) it follows that $h(V(H))=g(V(H)) \geq 5|R|+4(|V(G)|-|R|)=4|V(G)|+|R|$. Since $|R| \geq \gamma(G)$ it follows that $g(V(H)) \geq 4|V(G)|+\gamma(G)$. Therefore, $\gamma_{[4R]}(H) \geq 4|V(G)| + \gamma(G)$ and the desired equality follows.
\qednew

Having in mind the claim above, we now suppose there exists an approximation algorithm , say $B$, for the M4RDP with approximation ratio $(1 - \epsilon) \ln |V|$ for some fixed $\epsilon > 0$. Now, based on such Algorithm $B$, the following polynomial time approximation algorithm \texttt{DS-Approx} for solving the \texttt{DOMINATION} problem on a graph $G$ is presented. Let $\alpha=(1 - \epsilon) \ln |V|$ for some  $\epsilon > 0$ and  $k>0$ be an integer.

\begin{algorithm} \label{approxdomset}
		\begin{algorithmic}[1]
			\Require{A graph $G$ and an integer $k>0$. }
			\Ensure{A dominating  set $T$ of $G$.}\\
			\textbf{if} there exists a dominating set $T$ of $G$ with $|T| < k$  \textbf{then}\\
			\hspace{0.39cm}return $T$\\
			\textbf{else}\\
			\hspace{0.3cm} Obtain the graph $H$ from $G$ as describe before (Figure \ref{gm} as example) \\
			\hspace{0.3cm} Obtain a 4RDF $g$ of $H$ using algorithm $B$\\
			\hspace{0.3cm} Let $T = \{v_i : g(v_i) + g(b_i) \geq 5\}$ \\
             \hspace{0.3cm} return $T$\\
			\textbf{end if}
		\end{algorithmic}
    \caption{\texttt{DS-Approx}($G$)}
\end{algorithm}	

If there exists a  dominating set $T$ of $G$ with cardinality $|T| < k$, then it can be computed in polynomial time. If the answer $T$ given by the \texttt{DS-Approx} algorithm has cardinality $|T| < k$, then it is optimal.  Otherwise, let $T^\prime$ be a dominating set of $G$ with $|T^\prime|=\gamma(G)$, and let $f$ be a 4RDF of $H$  with $f(V(H)) = \gamma_{[4R]}(H)$. Clearly $f(V(H)) \geq k$ and $|T| \leq g(V(H)) \leq \alpha\cdot(f(V(H))) \leq \alpha\cdot(4|V(G)| + |T^\prime|)$ = $\alpha\cdot\left(1 + \frac{4|V(G)|}{|T^\prime|}\right)|T^\prime|$. Therefore, \texttt{DS-Approx} approximates a dominating set within a ratio of $\alpha\cdot\left(1 + \frac{4|V(G)|}{|T^\prime|}\right)$. If $\frac{1}{| T^\prime |} < \epsilon/4$, then the approximation ratio becomes $\alpha\cdot\left(1 + \frac{4n}{|T^\prime|}\right) < (1 - \epsilon)(1 + |V(G)|\epsilon) \ln |V(G)| = (1 - \epsilon^\prime) \ln |V(G)|$, where $\epsilon^\prime = |V(G)|\epsilon^2+ \epsilon-|V(G)|\epsilon$.
Thus, we deduce that \texttt{DS-Approx} approximates a minimum dominating set within $(1 - \epsilon^\prime) \ln |V|$.
Since \texttt{DS-Approx} is a  $(1 - \epsilon^\prime) \ln |V(G)|$-approximation algorithm, by Theorem \ref{inapprox} and the fact that $\ln (2|V(G)|) \approx \ln |V(H)|$ for $|V(G)| \rightarrow \infty$, unless $P = NP$, M4RDP cannot be approximated within a ratio of $(1 - \epsilon) \ln |V(G)|$ for any $\epsilon > 0$.
\end{proof}

\subsection{Approximation algorithm}

An approximation algorithm with approximation ratio of $ 1+\ln (\Delta(G)+1)$ for the \texttt{DOMINATION} problem on a graph $G$ is known from \cite{clrs}. Let \texttt{Dom-Approx} be one such algorithm. We next propose the approximation algorithm \texttt{4RD-Approx} to solve M4RDP by using the already known \texttt{Dom-Approx} algorithm. Notice that any QRDF $f$ of a graph $G$ can be represented by a $6$-tuple $(V_0,V_1,\dots,V_5)$ such that $V_i=\{v\in V(G)\,:\,f(v)=1\}$ with $0\le i\le 5$, and we consequently may write $f=(V_0,V_1,\dots,V_5)$.

\begin{algorithm}[ht]
	\caption{\texttt{4RD-Approx}($ G $)}\label{redalgo}
	\begin{algorithmic}[1]
		\Require{\mbox{A graph $G$. }}
		\Ensure{\mbox{A quadruple Roman dominating 6-tuple $f=(V_0,V_1,\dots,V_5)$ of $G$.}}\\
		$ S  \gets$ \texttt{Dom-Approx}($ G $) \\
		$ f \gets (V \setminus S, \emptyset, \emptyset, \emptyset, \emptyset,S ) $  \\
		return $ f. $
	\end{algorithmic}
\end{algorithm}

In this \texttt{4RD-Approx} algorithm, a dominating set $S$ of the given graph is first obtained by using the \texttt{Dom-Approx} algorithm, and then we define a function $f$ that assigns the value $5$ to the vertices of $S$, and the remaining vertices are assigned label $0$ under such $f$. Clearly, such $f$ is a 4RDF, and so $f(V(G)) \leq 5(1+\ln (\Delta(G)+1)) \gamma(G)\leq 5 (1+\ln(\Delta+1))\gamma_{[4R]}(G)$. Therefore, the following result is deduced.

\begin{theorem}\label{th30}
\texttt{4RD-Approx} is an approximation algorithm for M4RDP on a graph $G$ with approximation ratio $ 5 (1+\ln (\Delta(G)+1)) $.	
\end{theorem}

\noindent From Theorem \ref{th30} the result below follows.

\begin{corollary}\label{corapx}
	For bounded degree graphs, M4RDP is in APX.
\end{corollary}

\subsection{APX-completeness}

An optimization problem which is in APX and is APX-hard is known as APX-complete. We show that M4RDP is \textit{APX-hard} by giving an \textit{L-reduction} from \texttt{DOMINATION} on graphs with maximum degree 3 (DOM-3) \cite{apxdom}, which is a known APX-complete problem to M4RDP. We also need the following well known result that can be found in \cite{Haynes}.

\begin{theorem}{\em \cite{Haynes}}
\label{th:bound-dom}
For any graph $G$, $\gamma(G) \geq \frac{|V(G)|}{\Delta(G) + 1}$.
\end{theorem}

\begin{theorem} \label{apxcmtrd}
	For graphs with maximum degree 4, M4RDP is APX-complete.
\end{theorem}

\begin{proof}
From Corollary \ref{corapx}, it follows that M4RDP belongs to APX for bounded degree graphs. Next an L-reduction from DOM-3 to M4RDP is provided. An instance $G=(V,E)$ of DOM-3 is reduced to a corresponding instance $H=(V',E')$ of M4RDP as given in the proof of Theorem \ref{inapproxmr3dp}. Notice that if $\Delta(G)=3$, then it holds $\Delta(H)=4$. First, let $S^\prime$ be  a dominating set of $G$ with $|S^\prime|=\gamma(G)$, and let $f$ be a 4RDF of $H$ with weight $\gamma_{[4R]}(H)$, as described in the proof of Claim \ref{minimumcrdf}. By using Claim \ref{minimumcrdf},
\begin{equation}
\label{eq:f-S'}
f(V(H)) = 4|V(G)|+\gamma(G)=4|V(G)| + |S'|
\end{equation}
Also, by using Theorem \ref{th:bound-dom}, we know that $|V(G)|\le 4\gamma(G)=4|S'|$. Hence, we deduce that from \eqref{eq:f-S'},
\begin{equation}
    \label{eq:gamma-S'}
    f(V(H)) = 4|V(G)|+\gamma(G)=4|V(G)| + |S'|\le 17|S'|.
\end{equation}

Let $g$ be a 4RDF on $V(H)$. Similarly to the proof of Claim \ref{minimumcrdf}, in the proof of Theorem \ref{inapproxmr3dp}, it can be shown that the set $S = \{v_i : h(v_i) + h(b_i) \geq 5\}$ is a dominating set of $G$. Since $g(b_i) + g(v_i) \geq 4$, it follows that $|S| \leq g(V(H)) - 4|V(G)|$. Hence, by using \eqref{eq:f-S'}, we deduce
\begin{equation}
\label{eq:diff-f-g}
|S| - |S'|\leq g(V(H)) - 4|V(G)| - |S'| \leq g(V(H)) - f(V(H)).
\end{equation}
Finally, from the two inequalities \eqref{eq:gamma-S'} and \eqref{eq:diff-f-g}, we conclude that there exists an L-reduction from DOM-3 to M4RDP, which implies that M4RDP is APX-hard. Therefore, M4RDP is APX-complete.
\end{proof}

\section{Integer Linear Programming Formulation}\label{sec:ILP}

\noindent
Motivated by the Integer Linear Programming (ILP) formulations given for different domination problems \cite{ILP5,ILP3,ILP1,ILP2,ILP4}, in this section we present a basic  ILP formulation for M4RDP using six sets of binary variables given below for each vertex $u$ of a graph $G$ based on a 4RDF $h$ of $G$.

\medskip\noindent
$a_u =
\begin{cases}
1, &  h(u) = 0,\\
0, &  $otherwise$.
\end{cases}$
\hspace{1cm} $b_u =
\begin{cases}
1, &  h(u) = 1,\\
0, &  $otherwise$.
\end{cases}$
\hspace{1cm} $c_u =
\begin{cases}
1, &  h(u) = 2,\\
0, &  $otherwise$.
\end{cases}$\\
\smallskip \\
\hspace{0.5cm} $d_u =
\begin{cases}
1, &  h(u) = 3,\\
0, &  $otherwise$.
\end{cases}$
\hspace{1cm} $e_u =
\begin{cases}
1, &  h(u) = 4,\\
0, &  $otherwise$.
\end{cases}$
\hspace{1cm} $f_u =
\begin{cases}
1, & h(u) = 5,\\
0, &  $otherwise$.
\end{cases}$

\medskip
The  following is an ILP model for  M4RDP using six variables defined above.

\medskip\noindent
Determine:
\begin{equation} \label{equation51}
\min \left\{\sum_{u \in V}(b_u + 2c_u+3d_u+4e_u+5f_u) \right\}
\end{equation}
subject to the constraints:
\begin{equation} \label{equation52}
(1-(a_u + b_u + c_u+d_u)) + 3e_u + 4f_u + \displaystyle\sum_{v \in N[u]}b_v + 2c_v+3d_v+4e_v+5f_v \geq \\  4+\sum_{v \in N(u)}b_v + c_v+d_v+e_v+f_v  , \; u \in V
\end{equation}
\begin{equation} \label{equation53}
a_u + b_u+c_u+d_u+e_u+f_u \leq 1, \; u \in V
\end{equation}
\begin{equation} \label{equation54}
a_u, b_u, c_u, d_u, e_u, f_u \in \{0,1\}
\end{equation}

The weight of a 4RDF of $G$ is minimized using Equation \ref{equation51}. Constraint \ref{equation52} ensures that each vertex with label in the set $\{0,1,2,3\}$ is adjacent to vertex/vertices with required labels. The condition that each vertex is assigned a unique label is guaranteed by  constraint \ref{equation53} and the condition that variables are binaries is guaranteed by constraint  \ref{equation54}. The number of variables is $6|V(G)|$ and the number of constraints is $2|V(G)|$ in the proposed ILP model for M4RDP.

\section{Concluding remarks}

Once we have presented several contributions on the quadruple Roman domination of graphs, we have detected a few possible research lines that can be considered as continuation of our work, some of them are next stated.
\begin{itemize}
    \item We have  shown that the 4RDP is NP-complete for (star or comb) convex bipartite graphs, and clearly, this means that this happens for tree convex bipartite graphs in general. However, we wonder on whether there is a specific tree $T'$, such that the 4RDP can be efficiently solved for the class of $T'$ convex bipartite graph.
    \item Threshold graphs are special case of the well known co-graphs. Theorem \ref{th:threshold} shows that finding $\gamma_{4R}(G)$ can be easily computed when $G$ is a threshold graph. In this sense, we wonder on whether the same fact occur when $G$ is a co-graph.
    \item Theorem \ref{th:efficient} shows lower and upper bound for $\gamma_{4R}(G)$ in terms of $\gamma(G)$. We have given an infinite number of graphs where these bounds are achieved. Thus, it seems worthy of characterizing all the graphs for which such bounds are tight.
    \item In Section \ref{sec:ILP} we present an ILP formulation for the M4RDP. In connection with this, a natural step forward consist of making some implementations of the model, and introducing some heuristics to solve the M4RDP on random graphs for instance.
\end{itemize}

\section*{Acknowledgements}
We first would like to thank the reviewers for their careful reading of the manuscript and comments, which have improved the quality of the paper.

This investigation was completed while P.V.S. Reddy was visiting the University of C\'adiz, Spain, supported by Science and Engineering Research Board, DST, Government of India, Grant Number SIR/2022/001151. I.\ G.\ Yero has been partially supported by the Spanish Ministry of Science and Innovation through the grant PID2022-139543OB-C41. Moreover, I.G. Yero also thanks the support of the program ``Ayudas para la recualificaci\'on del sistema universitario espa\~{n}ol para 2021-2023, en el marco del Real Decreto 289/2021, de 20 de abril de 2021'' to make a temporary stay at Rovira i Virgili University, Spain.

\section*{Author contributions statement} All authors contributed equally to this work.

\section*{Conflicts of interest} The authors declare no conflict of interest.

\section*{Data availability} No data was used in this investigation.



\section*{Appendix}

\noindent
\textbf{Proof of Claim \ref{claim-2}:}
Suppose that the number of vertices of type $x_i$ or $x'_i$ having positive label is $m\geq 1$. Hence, the number of vertices of type $x_i$ or $x'_i$ with label equal to zero is $6r-m$. Thus, one of the following situations occurs.
\begin{itemize}
    \item[(i)] $g(x_i) \neq 0$ or $g(x'_i) \neq 0$, which means $g(x_i)+g(y_i) \geq 5$ or $g(x'_i)+g(y'_i) \geq 5$, respectively.
    \item[(ii)] $g(x_i) = 0$ or $g(x'_i) = 0$, which leads to $g(y_i) \geq 4$ or $g(y'_i) \geq 4$, respectively.
\end{itemize}
Now, let $a = |\{y_i\,:\,g(y_i) = 4 \; \text{and} \; 1 \leq i \leq 3r\}\cup \{y'_i\,:\,g(y'_i) = 4 \; \text{and} \; 1 \leq i \leq 3r\}|$. Clearly, for each $x_i$ with $g(x_i) = 0$ and $g(y_i) = 4$, and for each $x'_i$ with $g(x'_i) = 0$ and $g(y'_i) = 4$, there must be a vertex $C_j\in N(x_i)$ with $g(C_j) \geq 2$. Thus, the required number of vertices of type $C_j$ with $g(C_j) \geq 2$ is $\left\lceil \frac{a}{3} \right\rceil$. Notice that each $C_j$ has only three neighbors of type $x_i$, and that every $x'_i$ is adjacent to every $C_j$. Thus, this number $\left\lceil \frac{a}{3} \right\rceil$ is not influenced by the vertices of type $x'_i$. Consequently,
$$g(V(H_{r,s})) \geq 5+5m+4a+5(6r-m-a)+2\left\lceil \frac{a}{3} \right\rceil=28r+5+2r-\frac{a}{3}.$$
Since $m \geq 1$ and $a \leq 6r-m$ it follows that $2r-\frac{a}{3} >0$, and so, from the inequality above, we deduce $g(V) \geq 28r+5+2r-\frac{a}{3} > 26r+5$, which is a contradiction. Thus, the proof of Claim \ref{claim-2} is completed.\qednew

\bigskip
\noindent
\textbf{Proof of Claim \ref{claim-3}:}
From Claim \ref{claim-2}, it follows that  $g(y_i) \geq 4$ for every $1 \leq i \leq 3r$. Suppose there exist $y_k$ such that $g(y_k) = 5$, and let us count such vertices. Consider a vertex $x_i$ with $1 \leq i \leq 3r$. Since $g(x_i)=0$ (by Claim \ref{claim-2}) and $g(y_i) \geq 4$, we deduce that
$$g(y_i) + \sum_{x_i\in C_j} g(C_j) = \sum_{v\in N[x_i]} g(v) = g(N[x_i])\ge |\{y:y \in N(x_i) \; \text{and} \; g(y) \neq 0\}|+4.$$

Note that any $x_i$ such that $g(C_j)=0$ for every $C_j$ with $x_i\in C_j$, satisfies that $g(y_i)= 5$. Moreover, for any vertex $x_i$ such that $g(C_j)\ne 0$ for some $C_j$ with $x_i\in C_j$ it holds that either $g(C_j)\ge 2$, or ($g(C_j)=1$ and $g(y_i)=5$). For otherwise (that is $g(C_j)=1$ and $g(y_i)=4$), we have $g(N[x_i]) = \sum_{v\in N[x_i]} g(v)=g(y_i)+g(C_j)=5<6 = |\{y:y \in N(x_i) \; \text{and} \; g(y) \neq 0\}|+4$, which is a contradiction. In this sense, since $g$ has weight at most $26r+5$, $g(b)=5$, $g(y_i) \geq 4$ and $g(y'_i) \geq 4$ for every $1 \leq i \leq 3r$, the number of vertices of type $y_i$ such that $g(y_i)=5$ is at most $K\le 2r$. This means there must be at least $R\ge r$ vertices $y_j$ such that $g(y_j)=4$. But then, the corresponding vertices $x_j$ adjacent to such $y_j$ (with $g(y_j)=4$) must have neighbors $C_k$ with $g(C_k)\ge 2$. Altogether, this leads to say that $g$ has weight $g(V(H_{r,s}))\ge 12r+5+5K+4R+2\left\lceil R/3\right\rceil=5+12r+5(3r-R)+4R+2\left\lceil R/3\right\rceil=5+26r+r-\left\lceil R/3\right\rceil>5+26r$, since $R\ge r$, and this is a contradiction. Hence, we have that for every $y_i$, with $1 \leq i \leq 3r$, it follows $g(y_i) = 4$ as required. \qednew






\begin{thebibliography}{99}
\bibliographystyle{cas-model2-names}

\bibitem{trd} H.A. Ahangar, M.P. \'Alvarez, M. Chellali, S.M. Sheikholeslami, J.C. Valenzuela-Tripodoro, Triple Roman domination in graphs, \textit{Applied Math. Comp.} \textbf{391} (2021) Paper 125444.		
			
\bibitem{apxdom} P. Alimonti, V. Kann, Some APX-completeness results for cubic graphs, \textit{Theor. Comput. Sci.} {\bf 237} (2000) 123--134.
		
\bibitem{qrdp}	J. Amjadi, N. Khalili, Quadruple Roman domination in graphs, {\it Discrete Math. Algorithms Appl.} {\bf 4}(3) (2021) Paper 2150130.

\bibitem{Biggs} N. Biggs, Perfect codes in graphs, \textit{J. Combin. Theory Ser. B} \textbf{15} (1973) 289--296.

\bibitem{Bresar} B. Bre\v{s}ar, M.A. Henning, S. Klav\v{z}ar, D.F. Rall, \textit{Domination games played on graphs}. Berlin: Springer (2021).

\bibitem{ILP5} Q. Cai, N. Fan, Y. Shi, S.Yao, Integer linear programming formulations for double Roman domination problem, \textit{Optimization Methods and Software} \textbf{37}(1) (2022) 1--22.

\bibitem{Chartrand}  G. Chartrand, T.W. Haynes, M.A. Henning, P. Zhang, \textit{From domination to coloring: Stephen Hedetniemi's graph theory and beyond}. Springer Nature. (2019).

\bibitem{Chellali}  M. Chellali, T.W. Haynes, S.T. Hedetniemi, A.A. McRae, Roman \{2\}- Domination, \textit{Discrete Appl. Math.} \textbf{204} (2016) 22--28.

\bibitem{RD1} M. Chellali, N.J. Rad, S.M. Sheikholeslami,  L. Volkmann, Varieties of Roman Domination II, \textit{AKCE Int. J. Graphs Comb.} \textbf{17}(3) (2020) 966-–984.

\bibitem{RD2} M. Chellali, N.J. Rad, S.M. Sheikholeslami,  L. Volkmann, {\it Varieties of Roman Domination}, In: Haynes, T.W., Hedetniemi, S.T., Henning, M.A. (eds) Structures of Domination in Graphs. Developments in Mathematics, Springer, Cham, 66 (2021).

\bibitem{inapproxalg} M. Chleb{\'\i}k, J. Chleb{\'\i}kov{\'a}, Approximation hardness of dominating set problems in bounded degree graphs, \textit{Inf. Comput.} {\bf 206} (2008) 1264--1275.

\bibitem{Cockayne} E.J. Cockayne, P.A. Dreyer, S.M. Hedetniemi, S.T. Hedetniemi, Roman domination in graphs, \textit{Discrete Math.} \textbf{278} (2004) 11-22.

\bibitem{ILP3} P. Duraisamy, S. Esakkimuthu, Linear programming approach for various domination parameters, \textit{Discrete Math. Algor. Appl}. \textbf{13}(1) (2021) 2050096.
	
\bibitem{ci} M.R. Garey, D.S. Johnson, \textit{Computers and Interactability : A Guide to the Theory of NP-completeness}, Freeman, New York (1979).

\bibitem{hhh-1} T.W. Haynes, S.T. Hedetniemi, M.A. Henning, \textit{Domination in Graphs: Core Concepts}. Springer, New York, (2022).

\bibitem{hhh-2} T.W. Haynes, S.T. Hedetniemi, M.A. Henning (editors), \textit{Structures of domination in graphs} (Vol. 66). Cham: Springer (2021).

\bibitem{hhh} T.W. Haynes, S.T. Hedetniemi, M.A. Henning (editors), \textit{Topics in Domination in Graphs}, Switzerland:  Springer International Publishing, 2020.
	
\bibitem{Haynes2} T.W. Haynes, S. Hedetniemi and P. Slater, \textit{Domination in Graphs: Advanced Topics}. (Marcel Dekker, 1997).

\bibitem{Haynes} T.W. Haynes, S. Hedetniemi, P. Slater, \textit{Fundamentals of Domination in Graphs}, (CRC Press, 1998).	
	
\bibitem{compdiff1} M.A. Henning, A. Pandey, Algorithmic aspects of semitotal domination in graphs, {\it Theor. Comput. Sci.} {\bf 766} (2019) 46--57.

\bibitem{Hen-Yeo}  M.A. Henning, A. Yeo, \textit{Total domination in graphs}. New York: Springer (2013).
			
\bibitem{jiang}	W. Jiang, T. Liu, T. Ren, K. Xu, Two hardness results on feedback vertex sets, In FAW-AAIM, (2011) 233--243.	
		
\bibitem{qrdp1} Z. Kou, S. Kosari, G. Hao, J. Amjadi, N. Khalili, Quadruple Roman domination in trees, {\it Symmetry} {\bf 13} (2021) Paper 1318.

\bibitem{clrs} C. E. Leiserson, R.L. Rivest, T.H. Cormen, C. Stein, \textit{Introduction to algorithms}, (MIT press Cambridge, MA, 2001).

\bibitem{threshold} N. Mahadev, U. Peled, \textit{Threshold graphs and related topics}, Elsevier (1995).

\bibitem{ILP1} I. Marija, A mixed integer linear programming formulation for restrained Roman domination problem, Theory and Applications of Mathematics \& Computer Science TAMCS, 5, (2015).

\bibitem{ILP2} I. Marija, Improved integer linear programming formulation for weak Roman domination problem, \textit{Soft computing} \textbf{22}(19) (2018) 1432--7643.

\bibitem{planarx3c} A. Oganian, D. Josep, On the complexity of optimal microaggregation for statistical disclosure control, \textit{Statistical Journal of the United Nations Economic Commission for Europe} \textbf{18} (2001) 345--353.

\bibitem{Ore} O. Ore, Theory of Graphs, \textit{Amer. Math. Soc. Colloq. Publ.} \textbf{38} Providence, (1962).

\bibitem{ILP4} D. Pandiaraja, E. Shanmugam, R. Sivakumar, Linear programming formulation for some generalized domination parameters, \textit{Discrete Math. Algor. Appl}. \textbf{13}(4)
(2021) Paper 2150044.	

\bibitem{lr} C.H. Papadimitriou, M. Yannakakis, Optimization, approximation, and complexity classes, {\it J. Comput. Syst. Sci.} \textbf{43} (1991) 425--440.

\bibitem{Revelle2000} C.S. Revelle, K.E. Rosing, Defendens imperium romanum: a classical problem in military strategy, \emph{American Mathematical Monthly} \textbf{107} (7) (2000) 585--594.

\bibitem{Stewart1999} I. Stewart, Defend the Roman Empire, \emph{Scientific American} \textbf{28}(6) (1999) 136--139.

\bibitem{Valen} J.C. Valenzuela-Tripodoro, M. A. Mateos-Camacho, M. Cera and M.
P. Alvarez-Ruiz, Further results on the $[k]$-Roman domination in graphs, Bull. Iranian Math. Soc. \textbf{50} (2024) Paper 27.

\bibitem{west} D.B. West, \textit{Introduction to graph theory}, Upper Saddle River: Prentice Hall (2001).

\end{thebibliography}



\end{document}